\documentclass[11pt,a4]{amsart}
\makeatletter
\renewcommand{\section}{\@startsection{section}{1}{0pt}{20pt}{6pt}{\large\bfseries}}
\makeatother

\usepackage{natbib}
\usepackage{enumerate}

\makeatletter \@addtoreset{equation}{section} \makeatother

\theoremstyle{plain}
\newtheorem{theorem}{Theorem}[section]
\newtheorem{prop}[theorem]{Proposition}
\newtheorem{lem}[theorem]{Lemma}
\newtheorem{cor}[theorem]{Corollary}

\theoremstyle{definition}

\theoremstyle{remark}
\newtheorem{remark}[theorem]{Remark}

\newcommand{\R}{\mathbb{R}}
\newcommand{\C}{\mathbb{C}}
\newcommand{\N}{\mathbb{N}}
\newcommand{\D}{\mathfrak{D}}
\newcommand{\St}{P}

\newcommand{\Sp}{\mathbb{P}}
\newcommand{\A}{\mathbf{A}}
\newcommand{\Ut}{Q}
\newcommand{\Ul}{\mathbf{U}}
\newcommand{\Up}{\mathbb{Q}}

\newcommand{\Ee}{\textrm{E}}

\newcommand{\Hr}{\mathbf{H0}}
\newcommand{\Hn}{\mathbf{H1}}

\newcommand{\E}{\mathbb{E}}

\newcommand{\Hq}{\mathcal{H}}
\newcommand{\I}{\mathcal{I}}
\newcommand{\M}{\mathcal{M}}
\newcommand{\Ip}{\mathcal{I}_{\alpha,\psi}}
\newcommand{\Ia}{\I_{\alpha,\psi_{\gamma}}}

\newcommand{\It}{\I_{\alpha,\psi_{\theta}}}

\newcommand{\Ilq}[1]{\I_{\alpha,\psi}\left(\frac{q}{\chi};#1\right)}

\newcommand{\Itq}[1]{\I_{\alpha,\psi_{\theta}}\left(q+\theta_{\alpha};#1\right)}

\newcommand{\Nf}{\mathcal{N}}

\newcommand{\Ntq}[1]{\Nf_{\alpha,\psi,\theta}\left(\frac{q}{\chi};#1\right)}

\newcommand{\ann}{a_n(\psi;\alpha)}

\newcommand{\Id}[1]{{{\mathbb{I}}}_{\{#1\}}}

\begin{document}
\bibliographystyle{plain}

\title{$q$-invariant functions for some generalizations of the  Ornstein-Uhlenbeck  semigroup}
\author{P. Patie}
\address{Department of Mathematical Statistics and Actuarial Science\\ University of Bern\\
Sidlerstrasse, 5\\ CH-3012 Bern\\ Switzerland}
\thanks{I wish to thank an anonymous referee for very helpful comments and careful reading that led to improving the presentation of the paper. This work was partially carried out while I was visiting the Project OMEGA of INRIA at Sophia-Antipolis and the
Department of Mathematics of ETH Z\"urich. I would like to thank the members of both groups for their hospitality.}

\begin{abstract}
We show that the multiplication operator  associated to a fractional power of a Gamma random
variable, with parameter $q>0$, maps the convex cone of the $1$-invariant functions for a self-similar semigroup into the convex cone
of the $q$-invariant functions for the associated  Ornstein-Uhlenbeck (for short OU) semigroup. We also describe the harmonic
functions for some other generalizations of the OU semigroup. Among the various
applications, we characterize, through their Laplace transforms, the
laws of  first passage times above and overshoot for certain
two-sided $\alpha$-stable OU processes and also for spectrally
negative semi-stable OU processes. These Laplace
transforms are expressed in terms of a new family of power series
which includes the generalized Mittag-Leffler
functions and generalizes the family of functions introduced by Patie \cite{Patie-06c}.
\end{abstract}
\maketitle

\section{Introduction and main results}
Let $\Ee=\R, \R^+$ or $[0,\infty)$ and let $X$ be the realization
of $(\St_t)_{t \geq 0}$, a Feller semigroup on $\Ee$ satisfying, for $\alpha
> 0$, the
$\alpha$-self-similarity property, i.e.~for any $c>0$ and every $f \in \mathfrak{B}(\Ee)$, the space of bounded Borelian functions on
$\Ee$, we have the following identity
\begin{equation} \label{eq:ssp}
\St_{c^{\alpha}t}f(cx)= \St_{t}\left(d_cf\right) (x), \quad x \in \Ee,
\end{equation}
where $d_c$ is the  dilatation operator, i.e.~$d_cf(x)=f(cx)$. We
denote by $(\Sp_x)_{x\in \Ee}$ the family of probability measures of
$X$ which act on $D(\Ee)$, the Skorohod space of c\`adl\`ag
functions from $[0,\infty)$ to $\Ee$, and by $(F^X_t)_{t\geq0}$ its natural filtration.
We also mention that, throughout the paper, $\E$ stands for a reference expectation operator.
Moreover, $\A$ (resp.~$\D(\A)$) stands for
its infinitesimal generator (resp.~its domain). We have in mind the
following situations
\begin{enumerate}
\item $\Ee=\R$ and $X$ is  an $\alpha$-stable L\'evy process.
\item \label{en:ss}  $\Ee=\R^+$ or $[0,\infty)$ and $X$ is a $\frac{1}{\alpha}$-semi-stable processes in the terminology of Lamperti.
\end{enumerate}
More precisely, let $\xi$ be a L\'evy
process starting from $x \in \R$,
Lamperti \cite{Lamperti-72} showed  that the time change process
\begin{equation} \label{eq:lamp}
 X_t  = e^{\xi_{A_t}}, \: t \geq 0,
 \end{equation} where
 \begin{equation*}
 A_t=\inf
\{ u \geq 0; \: V_u := \int_0^u e^{\alpha \xi_s} \: ds> t \},
\end{equation*} is an
$\frac{1}{\alpha}$-semi-stable positive Markov process starting from $e^x$.  Further,  we
assume that $-\infty<b:=\E[\xi_1]<\infty$ and denote the
characteristic exponent of $\xi$ by $\Psi$. We also suppose that $\xi$
is not arithmetic (i.e.~ does not live on a discrete subgroup
 $r\mathbb{Z}$ for some $r>0$). Then, on the one hand, for $b>0$ (resp.~$b\geq0$ and $\xi$ is spectrally negative), it is plain that $X$ has infinite lifetime and we know from Bertoin and Yor
\cite[Theorem 1]{Bertoin-Yor-02-b} (resp.~\cite[Proposition
1]{Bertoin-Yor-02}), that the family of probability measures
$(\Sp_x)_{x>0}$ converges in the sense of finite-dimensional
distributions to a probability measure, denoted by $\Sp_{0}$, as $x
\rightarrow 0+$. On the other hand, if $b<0$, then it is plain that $\xi$ drift towards $-\infty$ and $X$ has a
finite lifetime which is $T^X_0=\inf\{u\geq0;\: X_{u-}=0, X_u=0\}$.
In this case, we assume that there exists a
unique $\theta
>0$ which yields the so-called Cram\'er condition
\begin{equation} \label{eq:cramer}
\E[e^{\theta \xi_1}] = 1.
\end{equation}
Then,
under the additional condition $0<\theta<\alpha$,
Rivero \cite{Rivero-06} showed that the minimal process $(X,T^X_0)$
admits an unique recurrent extension that hits and leaves $0$
continuously a.s.~and which is a $\frac{1}{\alpha}$-semi-stable
process on $[0,\infty)$. With a slight abuse of notation, we write
$(\Sp_x)_{x>0}$ for the family of laws of such a recurrent extension.  We
gather the different possibilities in the following.\\
\begin{itemize}
\item[$\Hr$.] $\Ee=[0,\infty)$ and either $b> 0$ or $b<0$ and \eqref{eq:cramer}
holds with $\theta<\alpha$. Moreover, if $\xi$ is spectrally negative, the case $b=0$ is also allowed.\\
\end{itemize}
 We will also use
the following hypothesis\\
\begin{itemize}
\item[$\Hn$.] $\Ee\subseteq[0,\infty)$. $\xi $ has finite
exponential moments of arbitrary positive orders,
i.e.~$\psi(m)<\infty$ for every $m\geq0$ where
$\psi(m)=-\Psi(-im)$.\\
\end{itemize}
 If $\xi$ is
spectrally negative, excluding the degenerate cases, then $\Hn$ holds and for $b \in \R$, we write
$\theta_0$ for the largest root of the equation $\psi(u)=0$. Then,
being continuous and increasing on $[\theta_0,\infty)$, $\psi$ has
a well-defined inverse function $\phi:[0,\infty)\rightarrow
[\theta_0,\infty)$ which is also continuous and  increasing.

 For $r>0$, we write $\St_t^r=e^{-rt}\St_t$. We say that a non-negative function, $\I_r$, is
$r$-excessive (resp.~$r$-invariant) for $\St_t$ if for any $t>0$,
\begin{equation*}
\St^r_t \I_r(x)\leq  \I_r(x),
\end{equation*}
(resp.~if we have $=$ in place of $\leq$) and $\lim_{t\downarrow
0}\St^r_t \I_r(x)= \I_r(x)$ pointwise.  Taking $r=1$ and writing simply
$\I=\I_1$, we have $
 \St^1_{t}\I(x) =  \I(x)$. The self-similarity property \eqref{eq:ssp} then yields
\begin{equation} \label{eq:hs}
\St^r_{t}\left(d_{r^{\frac{1}{\alpha}}}\I\right)(r^{-\frac{1}{\alpha}}x) =
 \I(x),
\end{equation}
which entails the identity $\I_r(x) =
\I(r^{\frac{1}{\alpha}}x)$  for all $ x \in  \Ee$ and $r>0$.
We denote the convex cone
of $1$-excessive (resp.~$1$-invariant) functions for $X$ by
$\mathfrak{E}(X)$ (resp.~$\mathfrak{I}(X)$).

For any $\lambda>0$, the Ornstein-Uhlenbeck (for short OU)
semigroup, $(\Ut_t)_{t\geq0}$, is defined, for $f \in \mathfrak{B}(\Ee)$, by
\begin{equation} \label{eq:ous}
\Ut_t f(x)= \St_{e_{\chi}(t)}\left(d_{e'_{\lambda}(-t)}
f  \right)(x), \quad x \in \Ee, \: t\geq0,
\end{equation}
where $e_{\lambda}(t)= \frac{e^{\lambda t}-1}{\lambda}$,
$\chi=\alpha \lambda$ and we write $v_{\lambda}(.)$ for the continuous
increasing inverse function of $e_{\lambda}(.)$. We mention that
such a deterministic transformation of self-similar processes traces
back to Doob \cite{Doob-42} who studied the generalized OU processes
driven by symmetric stable L\'evy processes. Moreover,
Carmona et al.~\cite[Proposition 5.8]{Carmona-Petit-Yor-98} showed that
$(\Ut_t)_{t\geq0}$ is a Feller semigroup with infinitesimal
generator, for $f\in \D(\A)$, given by
\begin{equation*}
\Ul f(x) =  \A f(x) - \lambda x f'(x).
\end{equation*}
Let $U$ be the realization of the Feller semigroup $(\Ut_t)_{t\geq0}$. It follows from \eqref{eq:ous} that
\begin{equation}\label{eq:defp}
U_t =e'_{\lambda}(-t)X_{e_{\chi}(t)}, \: t \geq0.
\end{equation}
We deduce, with obvious notation, that $T^U_0=v_{\chi}(T^X_0)$ a.s.. If $\Ee=\R$ (resp.~otherwise), we call $U$ a self-similar
(resp.~semi-stable) OU process. We denote by $(\Up_x)_{x>0}$ the family of probability
measures of a \emph{semi-stable} OU process. We deduce, from the Lamperti mapping \eqref{eq:lamp} and  the discussions above the following.
\begin{prop}
For any $x>0$, there exists a one to one mapping between the
law of a L\'evy process starting from $\log(x)$ and the law of a semi-stable OU process starting from $x$. More precisely, we have
\begin{equation} \label{eq:mou}
U_t = e^{-\lambda t }e^{\xi_{\bigtriangleup_t}},\: t<T^U_0,
\end{equation}
where $\bigtriangleup_{t} =\int_0^t U_s^{-\alpha} ds$. Note that for $b>0$, the previous identity holds for any $t\geq0$.

 Moreover,  if \eqref{eq:cramer} holds with $0<\theta<\alpha$, then the minimal process $(U,T^U_0)$
admits a recurrent extension which hits and leaves $0$
continuously a.s.~which is the OU process associated to
$(X,\Sp_x)$. We write its family of laws by $(\Up_x)_{x>0}$.
\end{prop}
Under the condition
$\Hn$, for any $\gamma>0$, we denote by $\Up^{(\gamma)}_x$  the law of the semi-stable OU process, starting at $x\in \R^+$,
associated to the L\'evy process having Laplace exponent
$\psi_{\gamma}(u)=\psi(u+\gamma)-\psi(\gamma), \: u\geq0$. \\We say that a probability measure $m$ is
invariant for $U$ if it satisfies, for any $f\in \mathfrak{B}(\Ee)$,
\begin{equation*}
 \int_{\Ee} \Ut_t f(x)m(dx)=\int_{\Ee} f(x)m(dx).
\end{equation*}
Finally,
let $G_q$ be a gamma random variable independent of $X$, with parameter $q>0$, whose law is given by $
\gamma(dr) = \frac{e^{-r}r^{q-1}}{\Gamma(q)} dr$.
We are now ready to state the following.
\begin{theorem} \label{thm:1}If $\Ee=\R $ or {\bf{H0}}
holds then  the Feller process $U$ is positively recurrent and its
unique invariant measure is $\chi \, \Sp_{0}(X_1\in dx)$.

Next, assume that $\I\in L^1(\gamma(dr))$.  For any $q
>0$, we introduce the function $\I(q;x)$ defined by
\begin{equation} \label{eq:inv}
\I(q;x) = \chi^{\frac{q}{\chi }}\E\left[\I\left(\left(\chi G_{\frac{q}{\chi
}}\right)^{\frac{1}{\alpha}}x\right)\right], \: x \in \Ee.
\end{equation}
Then, if  $\I\in \mathfrak{I}(X) \: (\textrm{resp.~}\mathfrak{E}(X))$ then $\I(q;x) \in \mathfrak{I}^q(U)\: (\textrm{resp.~}\mathfrak{E}^q(U))$.

Consequently, if $\I\in \mathfrak{I}(X)$, we have, for  any $q>0$,
\begin{equation} \label{eq:ts}
(1+\chi t)^{-\frac{q}{\chi}}\St_t \left(d_{(1+ \chi t)^{-\frac{1}{\alpha}}}\I \right)\left(q;x\right) = \I(q;x), \: x\in \Ee.
\end{equation}
\end{theorem}
\begin{remark}
\begin{enumerate}
\item We call the multiplication operator \eqref{eq:inv} associated to a fractional power of a Gamma random variable,  the {\it $\Gamma$-transform}.

\item The characterization of time-space invariant functions of
the form \eqref{eq:ts}, associated to self-similar processes, has
been first identified by Shepp \cite{Shepp-67} in the case of the
Brownian motion and by several authors for some specific
processes: Yor \cite{Yor-84} for the Bessel processes,
Novikov \cite{Novikov-83} and  Patie \cite{Patie-OU-06} for the one
sided-stable processes. Whilst in the mentioned papers, the
authors made used of specific properties of the studied processes
to  derive the time-space martingales, we provide a proof which is
based simply on the self-similarity property.
\end{enumerate}
\end{remark}

We proceed by investigating the process $Y$, defined, for any $x,\beta \in \R$  and $\xi_0=0$ a.s., by
\begin{equation} \label{eq:defy}
Y_t = e^{\alpha \xi_t} \left(x +  \beta  \int_0^t e^{-\alpha
\xi_{s}}ds\right),\: t\geq 0.
\end{equation}
We call $Y$ the L\'evy OU process. We mention that this generalization of the OU process is a specific instance of the continuous analogue of random recurrence equations, as shown by de Haan  and Karandikar \cite{Haan-Karandikar-89}. They have been also well-studied  by Carmona et al.~\cite{Carmona-Petit-Yor-97}, Erickson and Maller \cite{Erickson-Maller-05}, Bertoin et al.~\cite{Bertoin-Lindner-07} and by Kondo et al.~\cite{Kondo-Sato-06}. In \cite{Carmona-Petit-Yor-97}, it is proved that $Y$ is a homogeneous Markov process with respect to the filtration generated by $\xi$. Moreover, they showed, from the stationarity and the independency of the increments of $\xi$, that, for any fixed $t\geq0$,
\begin{equation*}
Y_t \stackrel{(d)}{=} xe^{\alpha \xi_t} + \beta  \int_0^t
e^{\alpha \xi_{s}}ds.
\end{equation*} Then, if $\E[\xi_1]<0$, they deduced that, as $t\rightarrow \infty$, $\xi_t \stackrel{(a.s)}{\rightarrow}- \infty$ and
$Y_t\stackrel{(d)}{\rightarrow} \beta  V_{\infty}=\int_0^{\infty}
e^{\alpha \xi_{s}}ds$. We refer to Bertoin and Yor \cite{Bertoin-Yor-05} for a thorough survey on the exponential functional of L\'evy processes. In the spectrally negative case, it is well know that the law of $V_{\infty}$ is
self-decomposable, hence absolutely continuous and unimodal. Moreover, under the additional assumption that $\theta<\alpha$, its law has been computed in term
of the Laplace transform by Patie \cite{Patie-06c}.
Now, we introduce the process $Z$ defined, for any $x\neq 0,\beta \in \R$  and $\xi_0=0$ a.s., by
\begin{equation} \label{eq:defz}
Z_t =  e^{\alpha \xi_t} \left(x + \beta \int_0^t e^{\alpha \xi_s} ds\right)^{-1}, \: t\geq 0.
\end{equation}
Before stating the next result, we introduce some notation.  Let $B$ be a Borel subset of $\Ee$ and we write $T^{U}_{B}$ for the first
exit time from  $B$ by $U$. With a slight abuse of terminology, we say that for any $x\in \Ee$, a non-negative function  $\Hq$  is a $(q\Delta,{B})$-harmonic function
for $(U,\Up_x)$ if
\begin{equation}
\E_x\left[e^{-q\bigtriangleup_{T^{U}_{B}}}\Hq(U_{T^{U}_{B}})\Id{T^{U}_{B}<T^{U}_0}\right]=\Hq(x).
\end{equation}
When $\Delta_t$ is replaced by $t$ in the previous expression, we simply say that $\Hq$ is a $(q,B)$-harmonic function  for $(U,\Up_x)$. We are ready to state the following.

\begin{theorem} \label{thm:2}
Set $\beta=\alpha \lambda x$ in \eqref{eq:defz}. Then, to a process $Z$ starting from $\frac{1}{x}$, with $x\neq 0$, one can associate a semi-stable OU process
$(U,\Up_1)$ such that
\begin{equation} \label{eq:y_u}
Z_{t} =
x^{-1} U^{\alpha}_{\bigtriangledown_t}, \: t<T_0^U,
\end{equation}
where $\bigtriangledown_t = \int_0^t Z_s ds$ and its inverse
is given by $\bigtriangleup_t=\int_0^tU_s^{-\alpha} ds$. Note that for $b>0$, the previous identity holds for any $t\geq0$. Consequently, with $x>0$, $Z$ is a Feller process on $(0,\infty)$.

Moreover, let $q>0$, $
0 \leq a<b\leq +\infty  $ and $ x \in (a,b)$. Then,  a
 $(q\Delta,T^U_{(ax,bx)})$-harmonic function for $(U,\Up_1)$ is a $(q,T^Z_{(a^{\alpha}, b^{\alpha})})$-harmonic function for the process $Z$ starting from $x^{-\alpha}$. Similarly, a  $(q\Delta,T^U_{(\frac{x}{b},\frac{x}{a})})$-harmonic for $(U,\Up_1)$ is a $(q,T^{\widehat{Y}}_{(a^{\alpha}, b^{\alpha})})$-harmonic function for the process $\widehat{Y}$ starting from $x^{\alpha}$, the L\'evy OU process
associated to the L\'evy process
$\widehat{\xi}=-\xi$, the dual of $\xi$ with respect to the Lebesgue
measure.

Finally, assume that $\Hn$ holds and  write $p_{q}(x)=x^{q}$ for  $x,q>0$. If the function $\Hq$ is $(\lambda \phi(q),B)$-harmonic function for $(U,\Up^{(\phi(q))}_x)$ then the function $p_{\phi(q)}\Hq$ is $(q\Delta,B)$-harmonic function for $(U,\Up_x)$.\\

\end{theorem}

\section{Proofs}
\subsection{Proof of Theorem \ref{thm:1}}
The description of the unique invariant measure is a refinement of
\cite[Proposition 5.7]{Carmona-Petit-Yor-98} where therein the proof
is provided for $\R$-valued self-similar processes and can be
extended readily for the $\R^+$-valued case under the condition
{\bf{H0}}, which ensures that $(X, \Sp_x)$ admits an entrance law at $0$.\\
Next, let us assume that $\I\in L^1(\gamma(dr))\cap \mathfrak{I}(X)$. We need to show that for any $q >0$, $e^{-qt}\Ut_t\I(q;x) = \I(q;x)$. For $x\in \Ee$, we deduce from the definition of $(\Ut_t)_{t\geq0}$ that
\begin{eqnarray*}
 e^{-qt}\Ut_t\I(q;x)&=&\frac{\chi^{\frac{q}{\chi}}}{\Gamma\left(\frac{q}{\chi}\right)}e^{-qt}\E_x\left[\int_0^{\infty}
\I\left((\chi r)^{\frac{1}{\alpha}}U_t\right)e^{-r} r^{\frac{q}{\chi}-1} dr\right] \\
&=&\frac{\chi^{\frac{q}{\chi}}}{\Gamma\left(\frac{q}{\chi}\right)} e^{-qt}\E_x\left[\int_0^{\infty}
\I\left((\chi r)^{\frac{1}{\alpha}}e'_{\lambda}(-t)X_{e_{\chi}(t)}\right)e^{-r} r^{\frac{q}{\chi}-1} dr\right].
\end{eqnarray*}
Using the change of variable $u=\chi e'_{\chi}(-t)r$, Fubini theorem and \eqref{eq:hs}, we get
\begin{eqnarray*}
e^{-qt}\Ut_t\I(q;x) &=& \frac{1}{\Gamma\left(\frac{q}{\chi}\right)}\E_x\left[\int_0^{\infty}
e^{-u e_{\chi}(t)}\I\left(u^{\frac{1}{\alpha}}X_{e_{\chi}(t)}\right)e^{-\frac{u}{\chi}} u^{\frac{q}{\chi}-1} du\right]\\
&=& \frac{1}{\Gamma\left(\frac{q}{\chi}\right)}\int_0^{\infty} \I\left(u^{\frac{1}{\alpha}}x\right)e^{-\frac{u}{\chi}} u^{\frac{q}{\chi}-1} du \\
&=& \I(q;x)
\end{eqnarray*}
where the last line follows after the change of variable $u=\chi r$.
The case  $\I\in L^1(\gamma(dr))\cap \mathfrak{E}(X)$ is obtained by following the same line of reasoning. The last
assertion is deduced from \eqref{eq:ous} and \eqref{eq:inv} by
performing the change of variable $u=v_{\chi}(t)$, with $v_{\chi}(t)=\frac{1}{\chi}\log(1+\chi t)$.\\

\subsection{Proof of Theorem \ref{thm:2}}
Setting  $\beta=\alpha \lambda x$, the Lamperti mapping \eqref{eq:lamp} yields
\begin{eqnarray*}
Z_{t} &=&x^{-1} e^{\alpha \xi_t} \left(1 + \alpha \lambda
 \int_0^t e^{\alpha
\xi_{s}}ds\right)^{-1}\\
&=&  x^{-1}\left(\frac{1}{(1+\alpha \lambda .)^{\frac{1}{\alpha}}}X_{.}\right)^{\alpha}_{V_t} \\
&=&
x^{-1}\left(U_{v_{\chi}(.)}\right)^{\alpha}_{V_t}
\end{eqnarray*}
where the last identity follows from \eqref{eq:ous}.
The proof of the
assertion \eqref{eq:y_u} is completed  by observing that
\begin{equation*}
(v_{\chi}\left(V_t\right))' = e^{\alpha \xi_t} \left(1
+ \alpha \lambda \int_0^t e^{\alpha \xi_{s}}ds\right)^{-1}.
\end{equation*}
Moreover, since the mapping $x \mapsto x^{\alpha}$ is a homeomorphism of $\R^+$, the Feller property follows from its invariance by "nice" time change of Feller processes, see Lamperti \cite[Theorem 1]{Lamperti-67}. We also obtain the following identities
\begin{eqnarray*}
T^Z_{(a,b)} &=& \inf \left\{ u\geq0;\: Z_u \notin (a,b)\right\} \\
&=& \inf \left\{ u\geq0;\: U^{\alpha}_{{\bigtriangledown_u}} \notin \left(ax,bx\right)\right\} \\
&=& \bigtriangleup \left(\inf \left\{ u\geq0;\: U_u \notin \left((ax)^{\frac{1}{\alpha}},(bx)^{\frac{1}{\alpha}}\right) \right\}\right).
\end{eqnarray*}
The characterization of the harmonic functions of $Z$ follows. The characterization of the harmonic functions of $\hat{Y}$ are readily deduced from the ones of $Z$ and the identity
\begin{equation*}
\hat{Y}_t = \frac{1}{Z_t}, \: t \geq0.
\end{equation*}
The proof of Theorem is then completed by using the following Lemma together with an application of the optional stopping theorem.
\begin{lem}
Assume that $\Hn$ holds, then
for $\gamma, \delta \geq 0$ and $x
>0$, we have
\begin{equation}\label{eq:absou}
d\Up^{(\gamma)}_{x} =
   \left(\frac{U_t}{x}\right)^{\gamma-\delta} e^{\lambda (\gamma-\delta)t-\left(\psi(\gamma)-\psi(\delta)\right) \bigtriangleup_t }   \,
   d\Up^{(\delta)}_{x},\quad \textrm{ on }
   F^U_{t} \cap \{t<T^U_0\} .
\end{equation}
Note that for
$b>0$ the condition "$\textrm{ on }
   \{t<T^U_0\}$"  can be omitted. For the particular case $\gamma=\theta$ and $\delta=0$, the absolute
continuity relationship  \eqref{eq:absou} reduces to
\begin{equation*}
d\Up^{(\theta)}_{x} =
   \left(\frac{U_t}{x}\right)^{\theta} e^{ \lambda \theta t}   \,
   d\Up_{x},\quad \textrm{ on }
F^U_{t} \cap \{t<T^U_0\}.
\end{equation*}
\end{lem}
\begin{proof}
We start by recalling that in \cite{Patie-06c}, the following power Girsanov
transform has been derived, under $\Hn$, for $\gamma, \delta \geq 0$
and $x
>0$, with obvious notation,
\begin{equation*}
d\Sp^{(\gamma)}_{x} =
   \left(\frac{X_t}{x}\right)^{\gamma-\delta} e^{- \left(\psi(\gamma)-\psi(\delta)\right) A_t }   \,
   d\Sp^{(\delta)}_{x},\quad \textrm{ on }
   F^X_{t}\cap \{t<T^X_0\}.
\end{equation*}

 The assertion
 \eqref{eq:absou} follows readily by time change and recalling that
\begin{equation*}
A_{e_{\chi}(t)} =\int_0^{t}
U^{-\alpha}_u du.
\end{equation*}
We complete the proof by recalling that for $b>0$, $U$ does not reach $0$ a.s..
\end{proof}

\section{Applications}
In this section, we illustrate our results to
some new interesting examples.

\subsection{First passage times and overshoot of stable OU processes}
Let $X$ be an $\alpha$-stable L\'evy process whose
characteristic  exponent satisfy, for $u\in \R$,
\begin{equation*}
\Psi(iu) = -c  |u|^{\alpha}\left(1-i\beta
{\rm{sgn}}(u)\tan\left(\frac{\alpha \pi }{2}\right)\right)
\end{equation*}
where  $1<\alpha<2$ and for convenience we take $c = \left(1+\beta^2
\tan^2\left(\frac{\alpha \pi }{2}\right) \right)^{-1/2}$. Then, we
introduce the constant $\rho=\Sp(X_1>0)$ which was evaluated by Zolotarev \cite{Zolotarev-57} as follows
\begin{equation*}
\rho = \frac{1}{2}
+\frac{1}{\pi\alpha}\tan^{-1}\left(\beta\tan\left(\frac{\alpha \pi
}{2}\right)\right).
\end{equation*}
Following Doney \cite{Doney-87}, we introduce, for any integers $k,l$, the class
$C_{k,l}$ of stable processes such that
\begin{equation*}
\rho+k = l\tilde{\alpha}
\end{equation*}
where $\tilde{\alpha}=\frac{1}{\alpha}$. For $m\in \N$, $x\in \R$
and $z\in \mathbb{C}$, introduce the function
\begin{equation*}
f_{m}(x,z) = \prod_{i=0}^m\left(z+e^{ix(m-2i)\pi}\right).
\end{equation*}
Next, we recall from the Wiener-Hopf factorization of L\'evy
processes due to Rogozin \cite{Rogozin-66}, that the law of the
first passage times $\tau^X_0$ and the over(under)shoot of $X$ at the level $0$ is described by the following
identities, for $\delta, r>0$ and $p\geq0$,
\begin{eqnarray*}
\int_0^{\infty}e^{- \delta x} \E_{-x}\left[e^{- r \tau^{X}_0-p
X_{\tau^X_0-}}\right] dx &=&
\frac{1}{\delta-p}\left(1-\frac{\Psi^+(-r^{\frac{1}{\alpha}}\delta)}{\Psi^+(-r^{\frac{1}{\alpha}}p)}\right)\\
\int_0^{\infty}e^{- \delta x} \E_x\left[e^{- r \tau^X_0-p
X_{\tau^X_0-}}\right] dx &=&
\frac{1}{\delta-p}\left(1-\frac{\Psi^-(r^{\frac{1}{\alpha}}\delta)}{\Psi^-(r^{\frac{1}{\alpha}}p)}\right)
\end{eqnarray*}
where $
\left(1-\Psi(\delta)\right)^{-1} = \Psi^-(\delta)\Psi^+(\delta)$.
Here $\Psi^+(\delta)$ (resp.~$\Psi^-(\delta)$) is analytic in $\Re(\delta)<0$ (resp.~$\Re(\delta)>0$)
 continuous and nonvanishing on $\Re(\delta)\leq0$ (resp.~$\Re(\delta)\geq0$).
 Doney \cite{Doney-87} computes the Wiener-Hopf factors for
stable processes in $C_{k,l}$ as follows
\begin{eqnarray*}
\Psi^+(z)&=&\frac{f_{k-1}(\alpha,(-1)^l(-z)^{\alpha})}{f_{l-1}(\tilde{\alpha},(-1)^{k+1}z)},\quad \mathfrak{Arg}(z) \neq 0,\\
\Psi^-(z)&=&\frac{f_{l-1}(\tilde{\alpha},(-1)^{k+1}z)}{f_{k}(\alpha,(-1)^lz^{\alpha})},\quad \mathfrak{Arg}(z) \neq -\pi,\\
\end{eqnarray*}
where $z^\beta$ stands for $\sigma^\beta e^{i\beta\phi}$ when
$z=\sigma e^{i\phi}$ with $\sigma>0$ and $-\pi< \phi\leq \pi$. Observe also that  $\Psi^+(-x^{\frac{1}{\alpha}})\sim  x^{-\rho}$ for large real $x$. Moreover,
using the fact that the function $\E_{x}\left[e^{- r \tau^X_0-p
X_{\tau^X_0}}\right], \: x \in \R$, is $r$-excessive for the semigroup of $X$, we deduce from the $\Gamma$-transform the following.
\begin{cor}
For any $q,\delta> 0$, $p\geq0$, and for any integers $k,l$ such that
 $ X \in C_{k,l}$, we have
\begin{eqnarray*}
 \int^0_{-\infty}\hspace{-0.2cm}e^{\delta x} \E_{x}\left[e^{- q \tau^{U}_0-p\:
U_{\tau^{U}_0-}}\right]dx  \hspace{-0.2cm} & = &\hspace{-0.2cm}\frac{1}{\delta-p}\left(
\chi^{\frac{q}{\chi}}-\frac{1}{\Gamma(\frac{q}{\chi})}
\int_0^{\infty}\frac{\Psi^+(-r^{\frac{1}{\alpha}}\delta)}{\Psi^+(-r^{\frac{1}{\alpha}}p)}e^{-\frac{r}{
\chi}}r^{\frac{q}{\chi}-1} dr\right)\\ \int_0^{\infty}\hspace{-0.2cm}e^{- \delta x}
\E_x\left[e^{- q \tau^{U}_{0}-p\: U_{\tau^{U}_{0}-}}\right]dx \hspace{-0.2cm} &=& \hspace{-0.2cm}
\frac{1}{\delta-p}\left(
\chi^{\frac{q}{\chi}}-\frac{1}{\Gamma(\frac{q}{\chi})}
\int_0^{\infty}\frac{\Psi^-(r^{\frac{1}{\alpha}}\delta)}{\Psi^-(r^{\frac{1}{\alpha}}p)}e^{-\frac{r}{
\chi}}r^{\frac{q}{\chi}-1} dr\right).
\end{eqnarray*}
\end{cor}

\subsection{First passage times of one-sided semi-stable- and L\'evy-OU
processes} We now fix $(\St_t)_{t\geq0}$ to be the semigroup of a spectrally negative
$\frac{1}{\alpha}$-semi-stable process $X$. $X$ is then associated via the Lamperti mapping \eqref{eq:lamp} to a spectrally negative L\'evy process, $\xi$,
which we assume to have a finite mean $b$. Its characteristic exponent $\psi$ has the well known L\'evy-Khintchine representation\begin{eqnarray}\label{eq:lap-levy}
\psi(u) = bu + \frac{\sigma}{2} u^2 + \int^{0}_{-\infty} (e^{u r} -1
-ur )\nu(dr),\: u\geq0,
\end{eqnarray}
 where $ \sigma \geq 0$ and the measure
$\nu$ satisfies the integrability condition $\int_{-\infty}^0 (r
\wedge r^2 )\:\nu(dr) <+ \infty$. Patie \cite{Patie-06c} computes the
Laplace transform of the first passage times above of $X$ as
follows. For any $r \geq 0$ and $0\leq x\leq a$, we have
\begin{eqnarray} \label{eq:lap_ssm}
 {\E}_x \left[e^{-r T^X_a } \right] &=&\frac{\Ip(rx^{\alpha})}{\Ip(ra^{\alpha})}
\end{eqnarray}
where the entire
function, $\Ip$, is given, for $\gamma \geq 0$ and $\alpha>0$, by
\begin{equation*}
\Ip(z)=\sum_{n=0}^{\infty} \ann z^{n}, \quad
 z
\in  \mathbb{C}
\end{equation*}
and
\begin{equation*}
\ann^{-1}=\prod_{k=1}^n \psi(\alpha k), \quad a_0=1.
\end{equation*}
Using the $\Gamma$-transform, we introduce the following power series
\begin{equation} \label{eq:f1}
\Ip(q;z)= \sum_{n=0}^{\infty}
 \ann
 (q)_{n} z^{n}
\end{equation}
where $(q)_{n} = \frac{\Gamma(q+n)}{\Gamma(q)}$ is the Pochhammer
symbol and  we have used the integral representation of the
gamma function $\Gamma(q)=\int_0^{\infty} e^{-r}r^{q-1}dr, \: \Re(q)>0$. By means of the following
asymptotic formula of ratio of gamma functions, see
e.g.~Lebedev \cite[p.15]{Lebedev-72}, for $\delta>0$,
\begin{eqnarray}
(z+ n)_{\delta} &=& z^{\delta}\left[1+\frac{\delta(2n+\delta-1)}{2z}
+ O(z^{-2})\right], \:  |\textrm{arg } z| < \pi-\epsilon,\:
\epsilon>0,
\end{eqnarray}
we deduce that $\Ip(q;z)$ is an entire function in $z$ and is analytic on the domain
$\{q\in \C; \: \Re(q)>-1\}$.
For $b<0$, we recall that there exists $\theta>0$ such that $\psi(\theta)=0$ and thus $\psi_{\theta}(u)=\psi(\theta+u)$. In this case, by setting  $\theta_{\alpha}=\frac{\theta}{\alpha}$, it is shown in \cite{Patie-06c} that there exists a positive constant
 $C_{\theta_\alpha}$ such that
\begin{equation*}
\Ip\left(x^{\alpha}\right)\sim C_{\theta_\alpha}
 x^{\theta}\It\left(x^{\alpha}\right) \quad {\textrm{ as  }} \: x\rightarrow
 \infty.
 \end{equation*}
 We also introduce the function $\Nf_{\alpha,\psi,\theta}(q;x^{\alpha})$ defined by
 \begin{equation} \label{eq:f2}
\Nf_{\alpha,\psi,\theta}(q;x^{\alpha}) =\Ip(q;x^{\alpha})-
C_{\theta_\alpha}x^{\theta}\frac{\Gamma(q+\theta_\alpha)}{\Gamma(q)}\Itq{x^{\alpha}},\quad \Re(x)\geq0.
\end{equation}
Moreover, if we assume  that
there exists $\beta \in [0,1]$ and a constant $a_{\beta}>0$ such that $\lim_{u \rightarrow
\infty} \psi(u)/u^{1+\beta}= a_{\beta}$, then $C_{\theta_{\alpha}}$ is characterized by
\begin{eqnarray*}
 C_{\theta_{\alpha}} = \left\{
\begin{array}{lll}
 &  \frac{\Gamma(1-\theta_\alpha)}{\alpha
}\frac{(\theta_{\alpha}-1)!}{\prod_{k=1}^{\theta_{\alpha}-1}\psi(\alpha
k)}, \: & \textrm{if }  \theta_{\alpha}  \textrm{ is a positive integer}, \\
 & \\
 & \frac{\Gamma(1-\theta_\alpha)}{\alpha
}a^{-\theta_{\alpha}}_{\beta}e^{E_{\gamma} \beta
\theta_{\alpha}}\prod_{k=1}^{\infty}e^{-\frac{\beta\theta_{\alpha}}{k}}\frac{(k+\theta_{\alpha})\psi(\alpha
k)}{k\psi(\alpha k+\theta_{\alpha})},& \textrm{otherwise}, \\
\end{array}
\right.
\end{eqnarray*}
where $E_{\gamma}$ stands for the Euler-Mascheroni constant.
 We recall, also from \cite{Patie-06c}, that, for $r,x\geq0$,
 \begin{eqnarray*}
{\E}_x \left[e^{-r T^X_{0}} \right]
&=&\Ip\left(rx^{\alpha}\right)-C_{\theta_\alpha}(r^{\frac{1}{\alpha}}x)^{\theta}\It\left(rx^{\alpha}\right).
\end{eqnarray*}
We
deduce from Theorems \ref{thm:1} and \ref{thm:2} the following.
\begin{cor} \label{cor:fptsn}
Let $q \geq 0$ and $0< x \leq a$. Then,
\begin{equation*}
 \E_x \left[e^{-q T_a^{U} } \right] =\frac{\Ilq{\chi x^{\alpha}}}
{\Ilq{\chi a^{\alpha}}}
\end{equation*}
and
\begin{equation*}
 \E_x \left[\left(1+\chi T_{(\alpha)}^{X}\right)^{-\frac{q}{\chi} } \right] =\frac{\Ilq{\chi x^{\alpha}}}
{\Ilq{\chi a^{\alpha}}}
\end{equation*}
where $T_{(\alpha)}^{X} = \inf\{u\geq0; \: X_u = a(1+\chi u)^{\frac{1}{\alpha}}\}$. We also deduce that
\begin{equation*}
\E_x\left[e^{-q \bigtriangleup_{T_a^{U}}}\Id{T_a^{U}<T^U_0}\right]=
\left(\frac{x}{b}\right)^{\gamma}\frac{\Ia\left(\frac{\gamma}{\alpha};\chi x^{\alpha}\right)}{\Ia\left(\frac{\gamma}{\alpha};\chi a^{\alpha}\right)}.
\end{equation*}
Moreover, assume $b>0$ and set $\beta=\alpha \lambda x$ and  $\gamma=\phi(q)$. Then,
\begin{equation*}
\E_{\frac{1}{x}}\left[e^{-q T_a^{Z}}\right]=
\left(\frac{1}{bx}\right)^{\gamma}\frac{\Ia\left(\frac{\gamma}{\alpha};\chi \right)}{\Ia\left(\frac{\gamma}{\alpha};\chi (ax)^{\alpha}\right)}, \: 0< \frac{1}{x} \leq a ,
\end{equation*}

\begin{equation*}
 \E_a \left[e^{-q T_x^{\widehat{Y}}} \right] =\left(\frac{x}{a}\right)^{\frac{\gamma}{\alpha}}\frac{\Ia\left(\frac{\gamma}{\alpha};\chi^{\frac{1}{\alpha}}\right)}{\Ia\left(\frac{\gamma}{\alpha};\left(\frac{\chi a}{x}\right)^{\frac{1}{\alpha}}\right)}, \: 0<x \leq a.
\end{equation*}
Finally,  if $b<0$ and $0<\theta<\alpha$, we have
\begin{equation*}
 \E_x \left[e^{-q T^U_0} \right] =\frac{\Ntq{\chi x^{\alpha}}}
{\Ntq{\chi x^{\alpha}}}.
\end{equation*}

\end{cor}
\begin{remark}
From the strong Markov property and the absence of positive jumps, we easily get that first passage times above for the processes $U$ and $Z$  are infinitely divisible random variables. Hence, we obtain from Corollary \ref{cor:fptsn}, that  the functions \eqref{eq:f1} and \eqref{eq:f2} are Laplace transforms, with respect to the parameter $q$, of infinitely divisible distributions concentrated on the positive real line.
\end{remark}
We end up by investigating some special cases which allow to make some connections between the power series introduced and some well-known or new special functions.
\subsubsection{The confluent hypergeometric functions}
 \label{sec:b} We first consider a
Brownian motion with drift $-\nu$, i.e.~ $\psi(u)=\frac{1}{2}u^2-\nu
u$. Setting $\alpha=2$, we have $\theta=2\nu$ and therefore we
assume $\nu<1$. Its associated semi-stable process is well known to
be a Bessel process of index $\nu$ and thus the associated
Ornstein-Uhlenbeck process is, in the case $n=2\nu+1 \in \N$,  the radial norm of
$n$-dimensional Ornstein-Uhlenbeck process. We get
\begin{equation*}
\I_{2,\psi}(x)=(x/2)^{\nu/2}\Gamma(-\nu+1){\rm{I}}_{-\nu}\left(\sqrt{2x}\right)
\end{equation*}
where $
{\rm{I}}_{\nu}(x)=\sum_{n=0}^{\infty}\frac{(x/2)^{\nu+2n}}{n!\Gamma(\nu+n+1)}$
 stands for the modified Bessel function of index $\nu$, see e.g.~\cite[5.]{Lebedev-72}, and
\begin{eqnarray*}
\I_{2,\psi}(q;x^2)&=&\Phi\left(q,1-\nu,\frac{x^2}{2}\right) \\
\I_{2,\psi_{2\nu}}(q;x^2)&=&\Phi\left(q+\nu,\nu+1,\frac{x^2}{2}\right)
\end{eqnarray*}
where
$\Phi(q,\nu,x)=\sum_{n=0}^{\infty}\frac{(q)_n}{(\nu)_n n!}x^n$
 stands for the confluent hypergeometric function of the first kind, see e.g.~\cite[9.9]{Lebedev-72}. Using the asymptotic behavior of the Bessel function
  \begin{equation*}
{\rm{I}}_{\nu}(x)\sim\frac{e^{x}}{\sqrt{2\pi x}} \quad  \textrm{ as } x
\rightarrow \infty,
\end{equation*}
we deduce that $C_{2\nu}=- \frac{\Gamma(-\nu)}{\Gamma(\nu)}$. Hence,
 \begin{eqnarray*}
\Nf_{\alpha,\psi_{2\nu}}(q;x^2)&=&\left(\Phi\left(q,1-\nu,\frac{x^2}{2}\right)+x^{2\nu}\frac{\Gamma(-\nu)\Gamma(q+\nu)}{\Gamma(\nu)\Gamma(q)}\Phi\left(q,1-\nu,\frac{x^2}{2}\right)\right)\\
&=&\frac{\Gamma(q)\Gamma(q+\nu)}{\Gamma(\nu)}\Lambda\left(q,\nu+1,\frac{x^2}{2}\right)
\end{eqnarray*}
where $\Lambda(q,\nu+1,\frac{x^2}{2})$ is the confluent hypergeometric of the second kind. We mention that, in this case,  the results of Corollary \ref{cor:fptsn} are well-known and can be found in Matsumoto and Yor \cite{Matsumoto-Yor-00} and in Borodin and Salminen \cite[II.8.2]{Borodin-Salminen-02}.

\subsubsection{Some generalization of the Mittag-Leffler function} \label{sec:ml}
Patie \cite{Patie-06-poch} introduced a new parametric
family of one-sided L\'evy processes which are characterized by the
following Laplace exponents, for any  $ 1<\alpha< 2$,
and $\gamma
>1-\alpha$,
\begin{equation} \label{eq:lap_poch}
 \psi_{\gamma}(u)=
\frac{1}{\alpha}\left((
u+\gamma-1)_{\alpha}-(\gamma-1)_{\alpha}\right).
\end{equation}
Its characteristic triplet are $\sigma=0$, $\nu(dy)=\frac{
\alpha(\alpha-1)}{\Gamma(2-\alpha)
}\frac{e^{(\alpha+\gamma-1)y}}{(1-e^{y})^{\alpha+1}}dy, y<0,$
and $ b_{\gamma}=
(\gamma)_{\alpha}(\Upsilon(\gamma-1+\alpha)-\Upsilon(\gamma-1))$ where
$\Upsilon(\lambda)=\frac{\Gamma'(\lambda)}{\Gamma(\lambda)}$ is the
digamma function. In particular, if $\gamma_0 $ denotes  the zero of
the function $\gamma \rightarrow b_{\gamma}  $, then for $\gamma
\geq\gamma_0\in (1-\alpha, 0)$, $b\geq 0$.\\
{\it The case $\gamma=0$.} \eqref{eq:lap_poch} reduces to $ \psi(u)=
\frac{1}{\alpha}(u-1)_{\alpha}$. Observe that $\theta=1$,
$\psi'(1)=\frac{\Gamma(\alpha)}{\alpha}$ and
\begin{equation*}
a_n(\psi;\alpha)^{-1}=\frac{\Gamma(\alpha(n+1))}{\Gamma(\alpha)},
\quad a_0=1.
\end{equation*}
The series \eqref{eq:f1} can be written as follows
\begin{eqnarray*}
\I_{1,\psi_1}(q;x)&=&\Gamma(\alpha)\M^q_{\alpha,\alpha}(\alpha
x)\\
\I_{1,\psi}(q;x)&=&\Gamma(\alpha-1)\M^q_{\alpha,\alpha-1}(\alpha
x)
\end{eqnarray*}
where
\begin{equation*}
 \M^q_{\alpha,\beta}(z)= \sum_{n=0}^{\infty}  \frac{(q)_n z^n}{\Gamma(\alpha n+\beta)}, \quad
 z
\in  \mathbb{C},
\end{equation*}
stands for the Mittag-Leffler function of parameter
$\alpha,\beta,q>0,$ which was introduced by Prabhakar \cite{Prabhakar-71}. Moreover, we have, see e.g.~\cite{Patie-06c},
\begin{equation*}
\sum_{n=0}^{\infty}\frac{z^{\alpha n}}{\Gamma(\alpha n+\beta)} \sim
\frac{1}{\alpha}e^{x}x^{1-\beta} l(x^{\alpha}) \quad {\rm{ as }} \:
x \rightarrow \infty,
\end{equation*}
with $l$ a slowly varying function at infinity.
Thus, $C_{\frac{1}{\alpha}}=\frac{\alpha}{\alpha-1}$ and
\begin{equation*}
\Nf_{\alpha,\psi_1}(q;x^{\alpha}) =
\M^q_{\alpha,\alpha-1}(x^{\alpha})-\frac{\alpha
x}{\alpha-1}\frac{\Gamma(q+\frac{1}{\alpha})}{\Gamma(q)}\M^q_{\alpha,\alpha}(x^{\alpha}).
\end{equation*}

As concluding remarks, we first mention that in the diffusion case,
i.e.~when $(U,\Up_x)$ is the Ornstein-Uhlenbeck process associated
to a Bessel process, see \ref{sec:b}, the law of the first passage
time above can be expressed as an infinite convolution  of
exponential distributions with parameters given by the sequence of
positive zeros of the confluent hypergeometric function, see
Kent \cite{Kent-80} for more details. Beside this case, we do not know
whether such a representation  is available. For instance, the
location of the zeros of the generalized Mittag-Leffler functions,
considered in the second example treated above,  is still an open
problem, see e.g.~Craven and Csordas\cite{Craven-Csordas-06}.


\begin{thebibliography}{10}

\bibitem{Bertoin-Lindner-07}
J.~Bertoin, A.~Lindner, and R.~Maller.
\newblock On continuity properties of the law of integrals of {L}\'evy
  processes.
\newblock {\em To appear in S\'eminaire de Probabilit\'es, Lecture
  Notes in Mathematics}, 2008.

\bibitem{Bertoin-Yor-02-b}
J.~Bertoin and M.~Yor.
\newblock The entrance laws of self-similar {M}arkov processes and exponential
  functionals of {L}\'evy processes.
\newblock {\em Potential Anal.}, 17(4):389--400, 2002.

\bibitem{Bertoin-Yor-02}
J.~Bertoin and M.~Yor.
\newblock On the entire moments of self-similar {M}arkov processes and
  exponential functionals of { L}\'evy processes.
\newblock {\em Ann. Fac. Sci. Toulouse Math.}, 11(1):19--32, 2002.

\bibitem{Bertoin-Yor-05}
J.~Bertoin and M.~Yor.
\newblock {Exponential functionals of L\'evy processes}.
\newblock {\em Probab. Surv.}, 2:191--212, 2005.

\bibitem{Borodin-Salminen-02}
A.N. Borodin and P.~Salminen.
\newblock {\em Handbook of Brownian Motion - Facts and Formulae}.
\newblock Probability and its Applications. Birkh\"auser Verlag, Basel,
  $2^{nd}$ edition, 2002.

\bibitem{Carmona-Petit-Yor-97}
Ph. Carmona, F.~Petit, and M.~Yor.
\newblock On the distribution and asymptotic results for exponential
  functionals of {L\'e}vy processes.
\newblock {\em In M. Yor (ed.) {E}xponential functionals and principal values
  related to Brownian motion. Biblioteca de la Rev. Mat. Iberoamericana}, pages
  73--121, 1997.

\bibitem{Carmona-Petit-Yor-98}
Ph. Carmona, F.~Petit, and M.~Yor.
\newblock Beta-{gamma} random variables and intertwining relations between
  certain {Markov} processes.
\newblock {\em Rev. Mat. Iberoamericana}, 14(2):311--368, 1998.

\bibitem{Craven-Csordas-06}
T.~Craven and G.~Csordas.
\newblock The {F}ox-{W}right functions and {L}aguerre multiplier sequences.
\newblock {\em J. Math. Anal. Appl.}, 314(1):109--125, 2006.

\bibitem{Haan-Karandikar-89}
L.~de~Haan and K.L. Karandikar.
\newblock Embedding a stochastic difference equation into a continuous-time
  process.
\newblock {\em Stochastic Process. Appl.}, 32(2):225--235, 1989.

\bibitem{Doney-87}
R.A. Doney.
\newblock On {Wiener-Hopf} factorisation and the distribution of extrema for
  certain stable processes.
\newblock {\em Ann. Prob.}, 15(4):1352--1362, 1987.

\bibitem{Doob-42}
J.L. Doob.
\newblock The {Brownian} movement and stochastic equations.
\newblock {\em Ann. of Math.}, 43(2):351--369, 1942.

\bibitem{Erickson-Maller-05}
K.B. Erickson and R.A. Maller.
\newblock Generalised {O}rnstein-{U}hlenbeck processes and the convergence of
  {L}\'evy integrals.
\newblock In {\em S\'eminaire de Probabilit\'es XXXVIII}, volume 1857 of {\em
  Lecture Notes in Math.}, pages 70--94. Springer, Berlin, 2005.

\bibitem{Kent-80}
J.T. Kent.
\newblock Eigenvalue expansions for diffusion hitting times.
\newblock {\em Z. Wahrsch. Verw. Gebiete}, 52:309--319, 1980.

\bibitem{Kondo-Sato-06}
H.~Kondo, M.~Maejima, and K.I. Sato.
\newblock Some properties of exponential integrals of {L}\'evy processes and
  examples.
\newblock {\em Electron. Comm. Probab.}, 11:291--303, 2006.

\bibitem{Lamperti-67}
J.~Lamperti.
\newblock On random time substitutions and the {F}eller property.
\newblock In {\em Markov Processes and Potential Theory (Proc Sympos. Math.
  Res. Center, Madison, Wis., 1967)}, pages 87--101. Wiley, New York, 1967.

\bibitem{Lamperti-72}
J.~Lamperti.
\newblock Semi-stable {M}arkov processes. {I}.
\newblock {\em Z. Wahrsch. Verw. Geb.}, 22:205--225, 1972.

\bibitem{Lebedev-72}
N.N. Lebedev.
\newblock {\em Special Functions and their Applications}.
\newblock Dover Publications, New York, 1972.

\bibitem{Matsumoto-Yor-00}
H.~Matsumoto and M.~Yor.
\newblock An analogue of {Pitman}'s {2M-X} theorem for exponential {Wiener}
  functionals, {Part I: A} time-inversion approach.
\newblock {\em Nagoya Math. J.}, 159:125--166, 2000.

\bibitem{Novikov-83}
A.~A. Novikov.
\newblock The martingale approach in problems on the time of the first crossing
  of nonlinear boundaries.
\newblock {\em Trudy Mat. Inst. Steklov.}, 158:130--152, 1981 (English version
  1983).

\bibitem{Patie-06c}
P.~Patie.
\newblock Infinitely divisibility of solutions to some semi-stable
  integro-differential equations and exponential functionals of {L\'evy}
  processes.
\newblock {\em Submitted}, 2007.

\bibitem{Patie-06-poch}
P.~Patie.
\newblock  Law of the exponential functional of a family of one-sided L\'evy processes via self-similar continuous state branching processes with immigration and the Wright hypergeometric functions.
\newblock {\em Submitted}, 2007.

\bibitem{Patie-OU-06}
P.~Patie.
\newblock Two sided-exit problem for a spectrally negative $\alpha$-stable
  {Ornstein-Uhlenbeck} process and the {W}right's generalized hypergeometric
  functions.
\newblock {\em Electron. Comm. Probab.}, 12:146--160 (electronic), 2007.

\bibitem{Prabhakar-71}
T.R. Prabhakar.
\newblock {A singular integral equation with a generalized Mittag- Leffler
  function in the kernel}.
\newblock {\em Yokohama Math. J.}, 19:7--15, 1971.

\bibitem{Rivero-06}
V.~Rivero.
\newblock Recurrent extensions of positive self similar {M}arkov processes and
  {C}ram\'er's condition {II}.
\newblock {\em Bernoulli}, 13(4):1053--1070, 2007.

\bibitem{Rogozin-66}
B.A. Rogozin.
\newblock On the distribution of functionals related to boundary problems for
  processes with independent increments.
\newblock {\em Theory Probab. Appl.}, 11:580--591, 1966.

\bibitem{Shepp-67}
L.A. Shepp.
\newblock {{A} first passage problem for the {W}iener process}.
\newblock {\em Ann. Math. Stat.}, 38:1912--1914, 1967.

\bibitem{Yor-84}
M.~Yor.
\newblock {On square-root boundaries for {B}essel processes and pole seeking
  {B}rownian motion}.
\newblock {\em Stochastic analysis and applications (Swansea, 1983). Lecture
  Notes in Mathematics}, 1095:100--107, 1984.

\bibitem{Zolotarev-57}
V.M. Zolotarev.
\newblock {Mellin-Stieltjes} transform in probability theory.
\newblock {\em Theory Probab. Appl.}, 2:433--460, 1957.

\end{thebibliography}
\end{document}